\documentclass{article}
\usepackage{latexsym, amsthm, amsmath, bbm}
\usepackage{amssymb}
\usepackage{amsfonts}
\usepackage{amstext}
\usepackage{graphicx}
\usepackage{multicol}
\usepackage{subfigure}
\usepackage{textcomp}
\usepackage{enumerate}
\usepackage{verbatim}
\usepackage{comment}
\usepackage{xcolor}
\usepackage{epstopdf}
\usepackage{float}
\usepackage{caption}
\usepackage{breqn}

\newtheorem{Theorem}{Theorem}
\newtheorem{lemma}{Lemma}

\newtheorem{definition}{Definition}

\theoremstyle{remark}

\theoremstyle{remark}

\newcommand{\beq}{\begin{equation}}
\newcommand{\beql}[1]{\begin{equation}\label{#1}}
\newcommand{\eeq}{\end{equation}}

\theoremstyle{remark}
\newtheorem{Example}{Example}
\newcommand{\R}{{\mathbb{R}}}

\newcommand{\Z}{{\mathbb{Z}}}

\DeclareMathOperator{\conv}{conv}

\DeclareMathOperator{\diam}{diam}
\DeclareMathOperator{\sing}{Sing}

\begin{document}

\title{Extension of the first mixed volume to nonconvex sets}

\author{Emmanuel Tsukerman \thanks{Supported by the National Science Foundation Graduate Research Fellowship under Grant No.  DGE 1106400.  Any opinion, findings, and conclusions or recommendations expressed
in this material are those of the authors(s) and do not necessarily reflect the views of the National Science
Foundation.} \\
\small University of California, Berkeley }

\maketitle

\begin{abstract}
We study the first mixed volume for nonconvex sets and apply the results to limits of discrete isoperimetric problems. Let $ M,N \subset \R^d$. Define $D_N (M)=\lim_{\epsilon \downarrow 0} \frac{|M+\epsilon N|-|M|}{\epsilon}$ whenever the limit exists. Our main result states that for a compact domain $M \subset \R^d$ with piecewise $C^1$ boundary and bounded $N \subset \R^d$, $D_N(M)=D_{\conv(N)}(M)$ and $D_N(M)=\int_{\text{bd }M} h_N(u_M(x)) \, d \mathcal{H}^{d-1}(x)$.
\end{abstract}

\section{Background}

 Minkowski's theorem on mixed volumes (see e.g.,  Chapter 5 of Schneider's text \cite{MR3155183}), states that the volume of a Minkowski sum of convex bodies can be written as a polynomial in the coefficients of that Minkowski sum, where the coefficients of the polynomial depend only on the convex bodies.  Specifically, let $\mathcal{K}^d$ denote the set of convex bodies in $\R^d$ -- that is, nonempty, compact, convex subsets of $\R^d$. We denote by $|S|$ the $d$-dimensional volume of $S \subset \R^d$. Then 

\begin{Theorem}\label{MixedVolumeThm}
Suppose $K_1, K_2, \dots, K_m \in \mathcal{K}^d$. For $\lambda_i \geq 0$ for $i=1,\ldots,m$, 
\begin{equation}
|\lambda_1K_1+\lambda_2 K_2+ \cdots + \lambda_m K_m| = \sum \lambda_{i_1}\lambda_{i_2} \cdots \lambda_{i_d} V\left(K_{i_1}, K_{i_2}, \dots, K_{i_d}\right)
\end{equation}
where the sum on the left hand side is the Minkowski sum, and the sum on the right hand side is over all multisets of size $d$ whose elements are in the set $\{1, 2, \dots, m\}$.  The  functions $V$ are nonnegative, symmetric, and depend only on the convex bodies $K_{i_1}, K_{i_2}, \dots, K_{i_d}$.  For a fixed $d$-dimensional convex body $K$, $V(\underbrace{K, K, \dots, K}_{d \text { times}}) = |K|$.
\end{Theorem}

We will be interested in generalizing the domain of the first mixed volume, $V(M,\ldots,M,N)$. 
In the literature, the special mixed volume 	
\beq
(M,K) \rightarrow V_1(M,K)=V(M,\ldots,M,K),
\eeq
 is known to have extensions to certain nonconvex sets $M$, important for several applications \cite{MR3155183}. For $M,K \in \mathcal{K}^d$,
\beql{lim}
V_1(M,K)=\frac{1}{d} \lim_{\epsilon \downarrow 0} \frac{|M+\epsilon K|-|M|}{\epsilon}
\eeq
and, taking $h_K$ to be the support function of $K$,
\beq
V_1(M,K)=\frac{1}{d}\int_{\mathbb{S}^{d-1}} h_K(u) S_{d-1}(M,du).
\eeq
The latter formula can be transformed into
\beql{bla}
V_1(M,K)=\frac{1}{d}\int_{\text{bd }M} h_K(u_M(x)) \, d \mathcal{H}^{d-1}(x),
\eeq
where $u_M(x)$ is the outer normal vector of $M$ at $x \in \text{bd }M$ and $\mathcal{H}^k$ is $k$-dimensional Hausdorff measure. Equation \eqref{bla} can be used to define $V_1(M,K)$ if $M$ is a compact set with a boundary which is a piecewise $C^1$ hypersurface (but $K$ remains a convex body). In this case, the limit relation \eqref{lim} remains valid \cite{MR1776095}. In this work, we further generalize to allow $K$ to be nonconvex and discuss the implications of this generalization.

To avoid confusion upon which definition of $V_1$ is being used, as not all definitions are equivalent when extended beyond convex bodies, we will introduce a different notation for $V_1$. This will emphasize both that we are no longer necessarily dealing with convex sets and the derivative-like nature of $V_1$. 

\begin{definition}
Let $ M,N \subset \R^d$. Define
\begin{equation}\label{deriv}
D_N (M)=\lim_{\epsilon \downarrow 0} \frac{|M+\epsilon N|-|M|}{\epsilon}
\end{equation}
whenever the limit exists.
\end{definition}

\begin{lemma}\label{LinearityLemma}
Suppose $K_1, K_2, K_3 \in \mathcal{K}^d$ and suppose $\alpha, \beta \in \R$.  Then
\begin{equation}
D_{\alpha K_2+\beta K_3}(K_1) = \alpha D_{K_2}(K_1)+\beta D_{K_3}(K_1).
\end{equation}
\end{lemma}

\begin{comment}
\begin{proof}[Proof of Lemma \ref{LinearityLemma}]
By definition we have
\begin{equation}
D_{\alpha u+\beta v}(P) = \lim_{\epsilon \to 0^+} \frac{\mu_d\left(P+\epsilon \left(\alpha u+\beta v\right)\right) - \mu_d(P)}{\epsilon} =  \lim_{\epsilon \to 0^+} \frac{\mu_d\left(P+\epsilon \alpha u+ \epsilon \beta v\right) - \mu_d(P)}{\epsilon} 
\end{equation}
From Theorem \ref{MixedVolumeThm}, we can see that
\begin{equation}
D_{\alpha u+\beta v}(P) = \alpha V(\underbrace{P, P, \dots, P}_{n-1 \text{ times}}, u) + \beta V(\underbrace{P, P, \dots, P}_{n-1 \text{ times}}, v) 
\end{equation}
where $V$ represents the function in the statement of Theorem \ref{MixedVolumeThm}.  Similarly, one can easily see that
\begin{align*}
D_u(P) &=  V(\underbrace{P, P, \dots, P}_{n-1 \text{ times}}, u) \\
 D_v(P) &=  V(\underbrace{P, P, \dots, P}_{n-1 \text{ times}}, v) 
 \end{align*}
and our Lemma is proved.
\end{proof}
\end{comment}

\section{Motivation from Discrete Isoperimetric Inequalities}

Our study of $D_\cdot(\cdot)$ was motivated by discrete isoperimetric inequalities. As mentioned in the previous section, the classical perimeter of a set can be found via Minkowski addition; when $M$ is convex and $u = B^d$ is the unit $d$-dimensional ball, $D_u(M)$ gives the perimeter of the set $M$. In discrete isoperimetric problems, one is interested in solving isoperimetric problems on graphs. The following definitions and results appear in \cite{TVdiscreteIso}. Let $G=(V,E)$ be a graph and let $\#|S|$ be the cardinality of a set $S$.

\begin{definition}
The \emph{vertex boundary} $\partial_V(S)$  of a set $S \subset V$ is the set of vertices in $V \backslash S$ which are adjacent to some vertex in $S$:
\begin{equation}
\partial_V(S) = \{v \in V\backslash S: (v,u) \in E \text{ for some } u \in S\}
\end{equation}
The \emph{edge boundary} $\partial_E(S)$ of a set $S \subset V$ is the set of edges $(u,v) \in E$ ``exiting'' the set $S$:
\begin{equation}
\partial_E(S) = \{(u,v) \in E: \left| \{u,v\} \cap S\right|= 1\}
\end{equation}
\end{definition}

A discrete isoperimetric problem is a problem of the form

\begin{equation}
\begin{aligned}
& \underset{S \subset V}{\text{minimize}}
& & \# |\partial (S)| \\
& \text{subject to}
& & \#|S|=n
\end{aligned}
\tag{DIP}\label{DIP}
\end{equation}
with $\partial(S)=\partial_V (S)$ in the case of a ``vertex-isoperimetric problem" and $\partial (S)=\partial_E(S)$ in the case of an ``edge-isoperimetric problem".

For certain families of graphs, it makes sense to consider the continuous limit of the problem. The limit of the discrete ``perimeter" turns out to be different from the ordinary perimeter and can be studied using the Brunn-Minkowski theory. We make this precise now.

\begin{definition}
A simple connected graph  $G = (V,E)$ is called a primitive lattice graph (PLG) if it satisfies the following: 
\begin{enumerate}
\item   \label{PLG1}$V $ is a lattice in $\R^d$.
\item \label{PLG2} The map $T_u:V \rightarrow V, \, T_u(v)=v+u$ is an automorphism of $G$ for every $u \in V$.
\item \label{PLG3} The edges $E$ are primitive vectors of the lattice $V$.
\end{enumerate}
\end{definition}
By an isomorphism, we will assume that $V = \Z^d \subset \R^d$. The assumption that $G$ is connected implies that the edge vectors give rise to a full rank lattice and that the convex hull $\conv(\{v_i\}_i)$ has full dimension. The former can be seen from the fact that there is a sequence of edges which leads from the origin to any standard basis vector. For the latter, both $v_i \in E$ and $-v_i \in E$, implying that the affine span of $\{v_i\}_i$ is a linear space. It has full dimension since the lattice has full rank.

\begin{Theorem}\label{DiscToContTheorem}
Suppose $G = (\Z^d, E)$ is a PLG graph with edge segments $\ell_1, \ell_2, \dots, \ell_k, k>0$.  Define the continuous edge-isoperimetric problem (CEIP) in $\R^d$ with boundary function $b$ arising from the EIP of $G$ by
\begin{equation}
\begin{aligned}
& \underset{S \subset \R^d}{\text{minimize}}
& & b(S)=\sum_{i=1}^k D_{\ell_i} S \\
& \text{subject to}
& & |S|=V 
\end{aligned}
\tag{CEIP}\label{opt-CEIP}
\end{equation}
\begin{enumerate}
\item \label{part2} A set $S \subset \R^d$ is a solution to (\ref{opt-CEIP}) if and only if it is homothetic to the zonotope
\begin{equation}
Z = \sum_{i=1}^k \ell_i.
\end{equation}
\item \label{part1} Under assumptions listed in \cite{TVdiscreteIso}, a sequence of optimal solutions to the EIP converges to the lattice points of a homothet of $Z$.
\end{enumerate}
\end{Theorem}

There is a similar continuous version of the vertex-isoperimetric problem. Here the boundary is replaced by $D_{\cup_i \ell_i} (S )$. In that present form, we are not able to apply Brunn-Minkowski theory as was done in the proof of Theorem \ref{DiscToContTheorem} to obtain a solution because $\cup_i \ell_i$ is not generally convex.

\section{Main Results}

Throughout this section, we will assume that $M,N \subset \R^d$ with $M$ a compact domain having a piecewise $C^1$ boundary $\partial M$ and $N$ is bounded. 

\begin{definition}
Let $y \in \partial M$ be a smooth point with outer normal ray $r=r_y(M)$. Given a bounded set $N \subset \R^d$ with center of mass at the origin, define the local expansion of $M$ at $y$ to be the set
\beq
Q_{y,M}(N)=r \cap (M+N)
\eeq
\end{definition}

\begin{figure}[hbtp]
\centering
\includegraphics[width=5in]{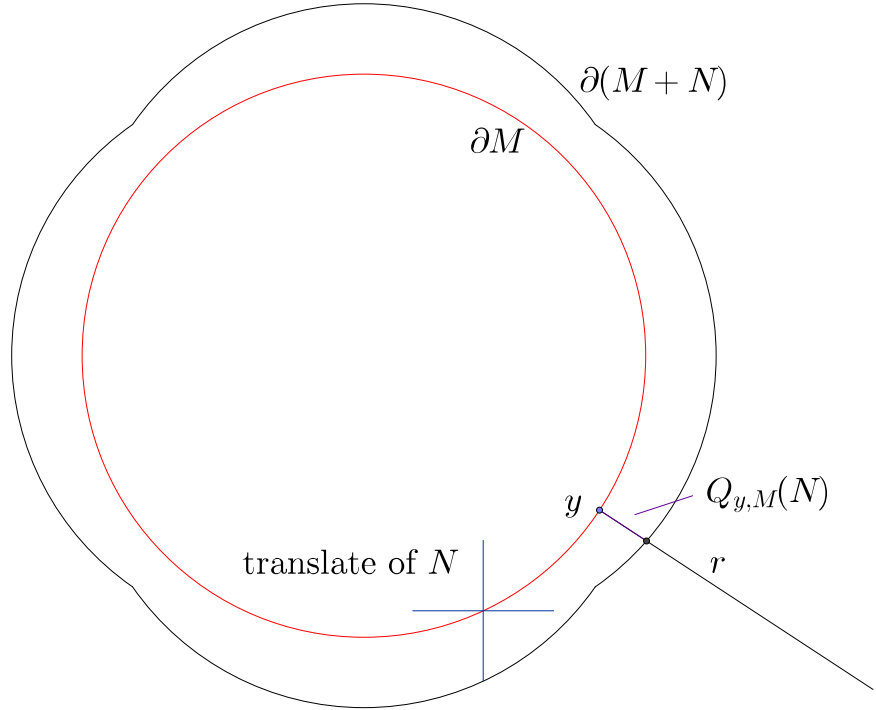}
\caption{\label{c1v2} $M$, $M+N$, and the local expansion $Q_{y,M}(N)$ to $M$ at $y$.}
\end{figure}

Although one may suggest a definition for a local expansion at singular points, we shall refrain from doing so.

\begin{lemma}\label{localexpansion}
Let $n(y)$ be the outer normal at a smooth point $y \in \partial M$. Then
\beq
\lim_{\epsilon\downarrow 0} \frac{|Q_{y,M}(\epsilon N)|}{\epsilon}=h_N(n(y)).
\eeq
\end{lemma}

\begin{proof}
It will be convenient to discuss first a situation in which $M$ is replaced by a half plane $H=\{x \in \R^d : x \cdot n \leq c\} $. Let $n$ be the unit outer normal at $y$. We have
\beq
H+\epsilon N =H+\epsilon \sup_{x \in N} x \cdot n=H+\epsilon h_N(n).
\eeq

\begin{figure}[H]
\centering
\includegraphics[width=5in]{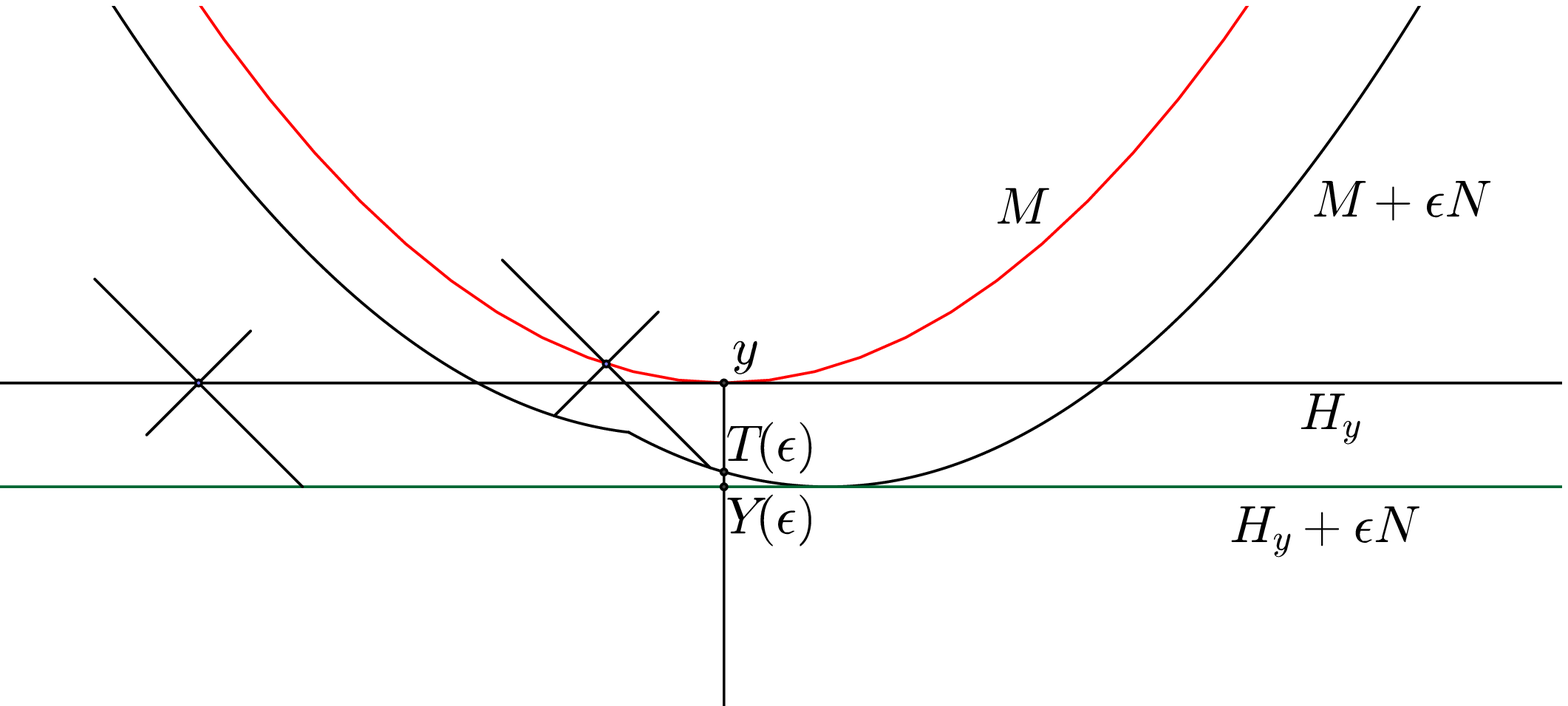}
\caption{\label{local_graph}}
\end{figure}

Now let $M$ be given. Set $w=\diam N$. Represent the boundary in a neighborhood of $y$ as the graph of a function $f:B_R^{d-1} \rightarrow \R$ with $(0,f(0))=y$. The function $f$ is differentiable at $0$. Take $\epsilon$ so that $2 \epsilon w>R>\epsilon w$. The condition $R>\epsilon w$ guarantees that translates of the form $x+\epsilon N$ for $x$ outside of the neighborhood of $y$ are too far to affect the local expansion at $y$. The condition $2 \epsilon w > R$ is needed to shrink the neighborhood of $y$.

Let $H_y$ be the supporting hyperplane to $M$ at $y$. Let $r$ be the normal ray to $y$ and let $Y(\epsilon)$ be the distance from the intersection of $\partial (H_y+ \epsilon N)$ with $r$ to $y$ and $T(\epsilon)$ be the distance between the intersection of $\partial (M+ \epsilon N)$ with $r$ and $y$ (see Figure \ref{local_graph}). 

From the case of the half-space, $\lim_{\epsilon \downarrow 0} \frac{|Y(\epsilon)|}{\epsilon}=\sup_{x \in N} x \cdot n$. For any $h \in B_R^{d-1}$, we have 
\beq
f(h)=f(0)+Df(0)h+r(h),  \quad \lim_{h\rightarrow 0} \frac{\|r(h)\|}{\|h\|}=0.
\eeq
Assume without loss of generality that the outer normal $n(y)$ points ``down", in the $-x_d$ direction. Let $(m,f(m))$ be a global maximum of $|f|$. Then 
\beq
T(\epsilon) \geq \sup_{x \in \epsilon N} x \cdot n+f(0)-f(m)=\epsilon \sup_{x \in  N} x \cdot n+f(0)-f(m)=\epsilon \sup_{x \in  N} x \cdot n-Df(0)m-r(m).
\eeq
We have $\|m\| \leq R <2 \epsilon w$. Therefore
\beq
\lim_{\epsilon \downarrow 0} \frac{|T(\epsilon)|}{\epsilon} \geq \lim_{\epsilon \downarrow 0} \frac{\epsilon \sup_{x \in  N} x \cdot n-Df(0)m-r(m)}{\epsilon} = \sup_{x \in  N} x \cdot n.
\eeq

\end{proof}

 \begin{Theorem} \label{dconv} For a compact domain $M \subset \R^d$ with piecewise $C^1$ boundary and bounded $N \subset \R^d$,
 \beql{dconvex}
D_N(M)=D_{\conv(N)}(M)
\eeq
and
 \beql{supportequation}
 D_N(M)=\int_{\text{bd }M} h_N(u_M(x)) \, d \mathcal{H}^{d-1}(x).
\eeq

 \end{Theorem}
 
 \begin{proof}
 By Lemma \ref{localexpansion}, the local expansions of $M$ with respect to $\epsilon N$ and with respect to $\epsilon \conv(N)$ converge in the limit $\epsilon \downarrow 0$. The difference in $D_{\cdot}(M)$ can only occur due to singularities and is bounded above by $D_{\conv(N)}(\sing(\partial M))$. Noting that $\conv(N) \subset B_\rho$, we can write
 \beql{diff}
D_{\conv(N)}(M)-D_N(M) \leq  D_{\conv(N)}(\sing(\partial M)) \leq D_{B_\rho}(\sing(\partial M)) , \quad \rho=\diam\conv(N).
\eeq
By hypothesis, the minkowski content of the singularities of $M$ is zero.
 \end{proof}
 
 In the notation of the background section, we have shown that if $M,N$ satisfy our running assumptions, then, even for nonconvex $N$,
 \beq
V_1(M,N)=\frac{1}{d}\lim_{\epsilon \downarrow 0} \frac{V_n(M+\epsilon N)-V_n(M)}{\epsilon}=\frac{1}{d}\int_{\text{bd }M} h_N(u_M(x)) d \mathcal{H}^{d-1}(x)
 \eeq

\section{Implications}

The function $D_{\conv(N)}(\cdot)$ is a continuous, translation invariant valuation on convex bodies. We have shown that $D_{N}(\cdot)$ agrees with $D_{\conv(N)}(\cdot)$ on a significant subset of this domain, i.e., on the convex bodies with piecewise $C^1$ boundaries. We will show now that, still, $D_{N}(\cdot)$ is not a continuous valuation on convex bodies.

To do so, we recall an important Theorem from the theory of valuations:

\begin{Theorem}{\cite[Theorem 6.3.5]{MR3155183}} Let $\phi$ be a translation invariant, continuous valuation on $\mathcal{K}^d$ with values in a real topological vector space. Then there are continuous, translation invariant valuations $\phi_0,\ldots,\phi_d$ on $\mathcal{K}^d$ such that $\phi_i$ is homogeneous of degree $i$ ($i=0,\ldots,d$) and
\beq
\phi(\lambda K)=\sum_{i=0}^d \lambda^i \phi_i(K) \text{ for } K \in \mathcal{K}^d\text{ and } \lambda \geq 0.
\eeq
In particular, $\phi=\phi_0+\cdots+\phi_d$.

\end{Theorem}

As a consequence, if $D_N$ is to be a continuous valuation on $\mathcal{K}^d$, applying $D_{\epsilon N}(\cdot)$ to a convex body should yield a polynomial in $\epsilon$. However, this is not the case.

\begin{Example}
Let $\ell_1=[(-1,0),(0,1)]$ and $\ell_2=[(0,-1),(0,1)]$ and $N=\ell_1 \cup \ell_2$. Let $K$ be the circle of radius $1$ centered at the origin. Refer to Figure \ref{c1v2} for helpful illustration.

The set $K+\epsilon \ell_1$ is the convex hull of the circles of radius $1$ centered at $(0,\pm \epsilon)$ and the set  $K+\epsilon \ell_2$ is the convex hull of the circles of radius $1$ centered at $(\pm \epsilon,0)$. The set $K+\epsilon N$ is the union of these two sets. The boundary naturally subdivides into $4$ arcs. The intersection points of the arcs have $x$-coordinates $\pm \frac{\epsilon+\sqrt{2-\epsilon^2}}{2}$ .

The volume of $K+\epsilon N$ is calculated from
\beq
\frac{|K+\epsilon N|}{2}=\int_{-1-\epsilon}^{-\frac{\epsilon+\sqrt{2-\epsilon^2}}{2}} \sqrt{1-(x+\epsilon)^2} \, dx+\int_{-\frac{\epsilon+\sqrt{2-\epsilon^2}}{2}}^{\frac{\epsilon+\sqrt{2-\epsilon^2}}{2}} \epsilon+\sqrt{1-x^2} \, dx+\int_{\frac{\epsilon+\sqrt{2-\epsilon^2}}{2}}^{1+\epsilon} \sqrt{1-(x-\epsilon)^2} \, dx.
\eeq
A computation in Mathematica shows that
\begin{dmath}
|K+\epsilon N|=\pi+2 \epsilon^2+2 \sqrt{2-\epsilon^2} \epsilon+\frac{1}{2} \sqrt{2-2 \epsilon \sqrt{2-\epsilon^2}} \epsilon+\frac{1}{2} \sqrt{2 \sqrt{2-\epsilon^2} \epsilon+2} \epsilon+\frac{1}{2} \sqrt{2-\epsilon^2} \sqrt{2-2 \epsilon \sqrt{2-\epsilon^2}}-\frac{1}{2}
   \sqrt{2-\epsilon^2} \sqrt{2 \sqrt{2-\epsilon^2} \epsilon+2}+{2 \sin ^{-1}(\frac{1}{2} (\epsilon-\sqrt{2-\epsilon^2}))}+2 {\sin ^{-1}(\frac{1}{2} (\sqrt{2-\epsilon^2}+\epsilon))}
\end{dmath}
\beq
=\pi +4 \sqrt{2} \epsilon+2 \epsilon^2-\frac{\sqrt{2} \epsilon^3}{3}-\frac{\epsilon^5}{20 \sqrt{2}}-\frac{\epsilon^7}{112 \sqrt{2}}-\frac{5 \epsilon^9}{2304 \sqrt{2}}-\frac{7 \epsilon^{11}}{11264 \sqrt{2}}-\frac{21
   \epsilon^{13}}{106496 \sqrt{2}}+O(\epsilon^{15}).
\eeq
Note that the first-order term in $\epsilon$ is equal to $D_{\conv(N)}(K)$ as expected.
\end{Example}

In particular, $D_{N\epsilon }(\cdot)$ is not polynomial in $\epsilon$. This also shows that $D_N(\cdot)$ differs from $D_{\conv(N)}(\cdot)$ on a convex set $K$ with boundary singularities having a nonzero minkowski content.

\bigskip
Another application of the main results is that in the continuous vertex-isoperimetric problem,  assuming $S$ is restricted to be a compact domain with piecewise $C^1$ boundary, we can applyequation \eqref{supportequation} in conjunction with the theory of Wulff shapes to find the solutions.

We recall the main result concerning Wulff shapes \cite{MR0031762, MR1130601, MR1898210}. Let $\partial^*S$ denote the reduced boundary of a measurable set $S$ and let $\Gamma:\mathbb{S}^{d-1} \rightarrow [0,\infty)$. The Wulff Theorem states that the variational problem

\begin{equation}
\begin{aligned}
& \underset{S \subset \R^d}{\text{minimize}}
& & \int_{\partial^* S} \Gamma(n_S(x)) \, d \mathcal{H}^{d-1} \\
& \text{subject to}
& & |S|=V, \text{ the perimeter of }S \text{ is finite,}
\end{aligned}
\tag{Wulff}\label{Wulff}
\end{equation}
has a unique solution up to translation and sets of measure zero, given by the Wulff set (or crystal of $\Gamma$) 
\beq
W_\Gamma:=\{x \in \R^d  :  x \cdot u \leq \Gamma(u), u \in \mathbb{S}^{d-1}\}.
\eeq

Suppose $S$ is a compact domain with piecewise $C^1$ boundary. By equation \eqref{supportequation} of Theorem \ref{dconv}, $D_{\cup_i \ell_i}(S)=\int_{\text{bd }S} h_{\cup_i \ell_i}(u_S(x)) \, d \mathcal{H}^{d-1}(x)$. Taking $\Gamma=h_{\cup_i \ell_i}$ and applying the Wulff Theorem yields

\begin{Theorem}\label{DiscToContVIPTheorem}
Suppose $G = (\Z^d, E)$ is a PLG graph with edge segments $\ell_1, \ell_2, \dots, \ell_k, k>0$.  Define the continuous vertex-isoperimetric problem (CVIP) in $\R^d$ with boundary function $b$ arising from the EIP of $G$ by
\begin{equation}
\begin{aligned}
& \underset{S \subset \R^d}{\text{minimize}}
& & b(S)=D_{\cup_i \ell_i} (S )\\
& \text{subject to}
& & |S|=V 
\end{aligned}
\tag{CVIP}\label{opt-CVIP}
\end{equation}
\begin{enumerate}
\item \label{part1vip} If $S$ is a compact domain solution to (\ref{opt-CVIP}) with piecewise $C^1$ boundary, then $S$ is homothetic to 
\beq
P:=\conv(\bigcup_i \ell_i),
\eeq
up to sets of measure zero.
\item \label{part2vip} Under assumption listed in \cite{TVdiscreteIso}, a sequence of optimal solutions to the VIP converges to the lattice points of a homothet of $P$.
\end{enumerate}
\end{Theorem}

\bigskip
{\bf Acknowledgments}.  
The author would like to thank Raman Sanyal, Rolf Schneider and Ellen Veomett for their helpful comments during the writing of this manuscript. Equation \eqref{dconvex} was suggested by conversation with Prof. Veomett. Reference \cite{MR0031762} on Wulff shapes was suggested by Prof. Schneider.

\bibliographystyle{alpha}
\bibliography{bibliography}

\end{document}